\documentclass{article}

\usepackage[utf8]{inputenc}
\usepackage{amsmath}
\usepackage{amsfonts}
\usepackage{amssymb}
\usepackage{amsthm}
\usepackage[english]{babel}
\usepackage{algorithm}
\usepackage{algorithmic}
\usepackage{subfigure}

\usepackage{graphicx}
\usepackage{hyperref}

\usepackage{pgf}
\usepackage{pgfplots}
\pgfplotsset{compat=newest}
\pgfrealjobname{perron_qve}

\newtheorem{theorem}{Theorem}
\newcommand{\norm}[1]{\left\Vert #1 \right\Vert}
\newcommand{\abs}[1]{\left\vert #1 \right\vert}
\newcommand{\R}{\mathbb R}
\newcommand{\ginv}{\#}
\DeclareMathOperator{\PV}{PV}

\newlength\plotheight
\newlength\plotwidth

\author{Beatrice Meini\thanks{Dipartimento di Matematica, 
Universit\`a di Pisa,
Largo B. Pontecorvo 5, 56127 Pisa, Italy 
({\tt meini@dm.unipi.it}).}\and Federico Poloni\thanks{Scuola Normale Superiore, Piazza dei 
Cavalieri 6, 56126 Pisa, Italy
({\tt f.poloni@sns.it}).}} \title{A Perron iteration for the
  solution of a quadratic vector equation arising in Markovian Binary Trees}
\begin{document}
\maketitle


\begin{abstract}
We propose a novel numerical method for solving a quadratic vector equation
arising in Markovian Binary Trees. The numerical method consists in a fixed
point iteration, expressed by means of the Perron vectors of a sequence of
nonnegative matrices. A theoretical convergence
analysis is performed. The proposed method outperforms the existing methods
for close-to-critical problems.
\end{abstract}

\section{Introduction}
In this paper we study the quadratic vector equation
\begin{equation}\label{eq:qve}
x=a+B(x\otimes x),
\end{equation}
where $a\in\R^n$, $B\in \R^{n\times  n^2}$ have nonnegative entries, the
symbol $\otimes$ denotes the Kronecker product, and  the unknown $x$ is
an $n$-dimensional vector.
The coefficients $a$ and $B$ are such that 
the vector $e=(1,1,\dots,1)^T$ is a solution of \eqref{eq:qve}.

Equation \eqref{eq:qve} arises in the study of
{M}{arkovian} Binary Trees (MBT), which are a particular family of branching processes used to model the growth of populations consisting of several types
of individuals, who may produce offsprings during their lifetime.
MBTs have applications in biology, epidemiology and also in telecommunication
systems. We refer to \cite{BeanKontoleonTaylor,HLR10} for definitions,
properties and applications.

One important issue related to MBTs is the computation of the extinction
probability of the population, which is the minimal
nonnegative solution $x^\ast\in\R^n$ of the quadratic vector equation
\eqref{eq:qve}. 

The MBT is called subcritical, supercritical or critical if the spectral
radius $\rho(R)$ of the matrix $R = B(e \otimes I + I \otimes e)$ is strictly
less than one, strictly greater than one, or equal to one, respectively. In
the subcritical and critical cases the minimal nonnegative solution is the
vector of all ones, while in the supercritical case $x^\ast\le e$, $x^\ast\ne
e$ (see \cite{HLR10} and \cite{an04}). Thus, only the supercritical case is of
interest for the computation of $x^\ast$.

Several numerical methods have been proposed for computing the vector
$x^\ast$. In \cite{BeanKontoleonTaylor} the authors propose two linearly
convergent algorithms, called \emph{depth} and \emph{order} algorithms. The \emph{thicknesses} algorithm, still linearly convergent, is
proposed in \cite{HLR10}. In \cite{HautphenneLatoucheRemiche} the authors
apply the Newton method, which has quadratic convergence. A modification of
Newton's method has been proposed in \cite{HVH10}.
All these methods have a probabilistic interpretation, and each of them
provides a sequence $\{x_k\}_k$ of nonnegative vectors, with
$x_0=(0,\ldots,0)^T$, which converges monotonically to the minimal nonnegative
solution $x^\ast$. A common feature of all these methods is that their
convergence speed slows down when the problem, while being supercritical,
is {\em close to critical}, i.e., the spectral radius of $R$ is close to one
and $x^\ast \approx e$. Moreover, the accuracy of the approximation deteriorates.

In this paper we write equation \eqref{eq:qve} in the form
$
x=a+b(x,x)
$
where $b(u,v):=B(u\otimes v)$ is the bilinear form defined by the matrix $B$.
If we set $y=e-x$, the latter equation  becomes
\begin{equation}\label{eqy0}
y=b(y,e)+b(e,y)-b(y,y).
\end{equation}
The sought solution $y^\ast$ of \eqref{eqy0} is $y^\ast=e-x^\ast$, where
$x^\ast$ is the minimal nonnegative solution of \eqref{eq:qve}.  In the
probability interpretation of Markovian Binary Trees, since $x^\ast$ is the
extinction probability, then $y^\ast=e-x^\ast$ is the survival probability.

Applying a functional iteration directly to \eqref{eqy0}, like Newton's
method, gives nothing new, since \eqref{eqy0} differs from \eqref{eq:qve} by a
linear change of variable. However, the new equation \eqref{eqy0} can be
rewritten as
\begin{equation}\label{eqHy0}
y=H_y y,
\end{equation}
where $H_y:=b(\cdot,e)+b(e-y,\cdot)$. The matrix $H_y$ is  nonnegative and
irreducible if $y<e$. In particular the solution $y^\ast$ is such that
$\rho(H_{y^\ast})=1$ and $y^\ast$ is the Perron vector of the matrix $H_{y^\ast}$.

This interpretation allows to design a new algorithm for computing $y^\ast$.
To this purpose, define the map $\PV(M)$ as the Perron vector of a nonnegative irreducible
matrix $M$, so that we may rewrite \eqref{eqHy0} as
\begin{equation} \label{eqPV0}
 y=\PV(H_y).
\end{equation}
The idea is to apply a fixed-point iteration to solve \eqref{eqPV0}, thus
generating a sequence $\{y_k\}_k$ of positive vectors such that 
$y_{k+1}=\PV(H_{y_k})$ and $y_k$ converges to $y^\ast$. A suitable normalization of the Perron vector,
consistent with the solution, is needed to obtain a well-posed iteration.
In this way we obtain a new iterative scheme, which is completely different
from classical functional iterations. Indeed, the proposed algorithm,
unlike known methods, fully exploits the fact that the solution  $x=e$ of the equation
\eqref{eqy0} is known. Moreover, the fixed-point
iteration at the basis of our algorithm relies on a new interpretation of the
solution $y^\ast$ in terms of the Perron vector.
These differences with respect to classical methods lead to great
improvements in the numerical solution
of MBTs which are close to critical.

We perform a convergence analysis of the fixed point iteration
$y_{k+1}=\PV(H_{y_k})$, by giving an expression to the Jacobian of the map
$y\to\PV(H_y)$. This expression allows to derive a local convergence
result. Moreover, most importantly,
we prove
that, although the convergence of the
method is linear, the speed of convergence \emph{increases} as the problem gets
close to critical. In the limit case of a critical problem, the convergence
becomes superlinear. This nice behavior is opposite to the one of Newton's
method, whose speed of convergence is sublinear in the supercritical case,
and becomes linear in the critical case.

A wide numerical experimentation confirms our theoretical analysis.  For
far-from-critical problems the standard techniques are preferable, while for
close-to-critical problems our method outperforms the existing ones.

The paper is organized as follows. In Section~\ref{assume} we state our
assumptions on the problem. In Section~\ref{optim} we rewrite the vector
equation in terms of an equation for the vector $y=e-x$ and discuss the
properties of the equation obtained in this way. The new algorithm,
based on a Perron iteration, is presented in Section~\ref{perron}. The
theoretical convergence analysis is performed in Section~\ref{sec:con}. Finally, in
Section~\ref{numer} we present the results of the numerical
experiments. Conclusions and open issues are addressed in Section~\ref{conclude}.

\section{Assumptions on the problem}\label{assume}
Let $a\in\R^n$, $B\in \R^{n\times  n^2}$ have nonnegative entries, and
consider 
the quadratic vector equation \eqref{eq:qve}
where it is assumed that the vector $e=(1,1,\dots,1)^T$ is a solution.  Let
$x^\ast\in\R^n$ be the minimal nonnegative solution of \eqref{eq:qve}, i.e., $x^\ast\le
x$ for any other nonnegative solution, where the semi-ordering is
component-wise. A unique solution $x^\ast$ exists, according to the results
of \cite[Section V.3]{an04}.

We assume that $\rho(R)>1$, where 
\[
R = B(e \otimes I + I \otimes e).
\] 
Under
this assumption $x^\ast\le e$, $x^\ast\ne e$ (see \cite{HLR10} and
\cite{an04}). It is worth pointing out that if $\rho(R)= 1$, then $x^\ast=e$,
therefore as $\rho(R)$ is greater than 1 and gets closer to 1, then $x^\ast$
approaches to the vector of all ones.

We introduce the bilinear operator 
\[
b(\cdot,\cdot):\R^n\times \R^n\to \R^n
\]
defined as
\[
b(u,v)=B(u\otimes v)
\]
and rewrite  \eqref{eq:qve} as
\begin{equation}\label{mbt}
x=a+b(x,x).
\end{equation}
We assume that for the minimal solution $x^\ast$ of \eqref{mbt} it holds
$0<x^\ast<e$, and that the Jacobian of the map $x\to x-a-b(x,x)$ at $x^\ast$,
i.e., $I-b(x^\ast,\cdot)-b(\cdot,x^\ast)$, is a nonsingular irreducible
M-matrix.  Since irreducibility is only determined by the nonnegativity
pattern of $b(\cdot,\cdot)$, the irreducibility condition is equivalent to
requiring that $b(e,\cdot)+b(\cdot,e)$ is irreducible. Notice that the latter
is just another notation to represent the matrix $R$.

Moreover, we may assume that $e^Tb(e-x^\ast,e-x^\ast)> 0$, otherwise
$b(e-x^\ast,e-x^\ast)=0$ (since $B\geq 0$), and the problem is trivial
since it becomes a linear problem.

\section{The optimistic equation}\label{optim}
A property of equation~\eqref{mbt} which has not been exploited so far in the
existing literature
is that $x=e$ is a solution.
If we set $x=e-y$, by using the bilinearity of the operator $b(\cdot,\cdot)$
and the property that $e=a+b(e,e)$, equation \eqref{mbt} can be rewritten as 
\begin{equation}\label{eqy}
y=b(y,e)+b(e,y)-b(y,y).
\end{equation}
The trivial solution is $y=0$, which corresponds to $x=e$.
We are interested in the nontrivial solution
$0<y^\ast<e$, which gives the sought solution $x^\ast=e-y^\ast$.

In the probabilistic interpretation of Markovian Binary Trees, $x^\ast$ is the
extinction probability, thus $y^\ast=e-x^\ast$ is the survival probability,
i.e., $y^\ast_i$ is the probability that a colony starting from a single
individual in state $i$ does not become extinct in a finite time.
For this reason, we refer to \eqref{eqy} as to the {\em optimistic
  equation}.

Notice that \eqref{eqy} admits the following probabilistic
interpretation. The term $b(y,e)$ represents the probability that the original
individual $\mathcal M$ (for ``mother'') spawns an offspring $\mathcal F$ (for
``first-born''), and after that the colony generated by the further offsprings
of $\mathcal M$, excluding $\mathcal F$, survives. The term $b(e,y)$
represents the probability that $\mathcal M$ spawns $\mathcal F$, and the colony generated by
$\mathcal F$ survives. The term $b(y,y)$ represents the probability that
$\mathcal M$ spawns
$\mathcal F$, and after that both their colonies survive. Thus \eqref{eqy} follows by
the well-known \emph{inclusion-exclusion} principle
\[
\mathbb P[\text{$\mathcal M$ or $\mathcal F$ survives}]=\mathbb
P[\text{$\mathcal M$ survives}]+ \mathbb P[\text{$\mathcal F$ survives}]-
\mathbb P[\text{both $\mathcal M$ and $\mathcal F$ survive}],
\]
where $\mathbb P[X]$ denotes the probability of the event $X$.

Equation \eqref{eqy} can be rewritten as
\begin{equation}\label{eq:pv}
y=H_y y
\end{equation}
where
\begin{equation}\label{eq:hy}
H_y=b(\cdot,e)+b(e,\cdot)-b(y,\cdot).
\end{equation}
Notice that $H_y$ is the sum of a fixed matrix and a
matrix that depends linearly on $y$. Therefore the quadratic operator on the
right-hand side of \eqref{eqy} is
``factored'' as the product of a matrix which depends on $y$, and $y$.

An important property is that $H_y$ is a nonnegative irreducible matrix,
whenever $y<e$. Therefore, by the Perron-Frobenius theorem \cite{varga}, if
$y<e$, $H_y$ has a positive eigenvalue $\lambda_y=\rho(H_y)$, the so-called
\emph{Perron value}, and to $\lambda_y$ corresponds a positive eigenvector
$w_y$, unique up to a multiplicative constant, the so-called \emph{Perron
  vector}, so that $H_y w_y=\lambda_y w_y$.  Therefore the sought solution
$y^\ast$ can be interpreted as the vector $0<y^\ast<e$ such that
$\rho(H_{y^\ast})=1$ and $y^\ast$ is a Perron vector of $H_{y^\ast}$.

It is worth pointing out that this interpretation of $y^\ast$ in terms of the
Perron vector allows to keep away from the trivial solution $y=0$ of
\eqref{eq:pv}, since the Perron vector has strictly positive elements.

The formulation of the quadratic vector equation in terms of the Perron vector
allows to design a new algorithm for its solution.

\section{The Perron iteration}\label{perron}
If we set up a fixed-point iteration or a Newton method for $y$ based on
\eqref{eqy}, we get the traditional fixed-point iterations and Newton methods
for MBTs \cite{HautphenneLatoucheRemiche}, since what we have done is simply a
linear change of variables. Instead, we exploit the fact that 
$y^\ast$ is a Perron vector of the nonnegative irreducible matrix
$H_{y^\ast}$ (compare \eqref{eq:pv}). 

To this purpose we introduce the operator
\[
u=\PV(X)
\]
which returns the Perron vector $u$ of the irreducible nonnegative matrix $X$.

Thus, we can devise a fixed-point
iteration to compute the solution $y^\ast$ by defining the sequence of vectors
\begin{equation} \label{eq:PVi}
 y_{k+1}=\PV(H_{y_k}),~~~k=0,1,2,\ldots,
\end{equation}
starting from an initial approximation $y_0$.  In order to define uniquely the
sequence $\{y_k\}$, we need to impose a normalization for the Perron vector,
which is uniquely defined up to a multiplicative constant. A possible choice
for the normalization is imposing that the residual of \eqref{eqy} is
orthogonal to a suitable vector $w\in \mathbb R^n$, i.e.,
\begin{equation}\label{nor}
w^T(y_{k+1}-b(y_{k+1},e)-b(e,y_{k+1})+b(y_{k+1},y_{k+1}))=0.
\end{equation}
Clearly, this normalization is consistent with the solution of \eqref{eqy}. We choose $w$ as the left Perron vector of the matrix $b(\cdot,e)+b(e,\cdot)$; the rationale for this choice is discussed in Section~\ref{sec:con}.

Given a Perron vector $u$ of $H_{y_k}$, the
equation to compute the normalization factor $\alpha$ such that $y_{k+1}=\alpha
u$ satisfies \eqref{nor} reduces to
\[
 \alpha w^T u = \alpha w^T b(u,e) + \alpha w^T b(e,u) - \alpha^2 w^T b(u,u),
\]
whose only non-zero solution is
\[
 \alpha=-\frac{w^T(u-b(u,e)-b(e,u))}{w^Tb(u,u)}.
\]
Notice that the solution $\alpha=0$ corresponds to the trivial solution $y=0$ ($x=e$), which we want to avoid.

The $\PV(\cdot)$ operator is defined on the set of irreducible nonnegative
matrices. If $y<e$, then the
matrix $H_y$ is nonnegative irreducible, therefore the sequence $y_k$ generated
by \eqref{eq:PVi} is well defined if $y_k<e$ for any $k$.
  
In Section \ref{sec:con} we show that the iteration \eqref{eq:PVi} is
locally convergent. Therefore, if $y_0$ is quite close to $y^\ast$, one can
expect that $y_k<e$ for any $k$.  In the case where $H_{y_k}$ is not a
nonnegative irreducible matrix, we can define $y_{k+1}$ as an eigenvector
corresponding to the eigenvalue of $H_{y_k}$ having maximal real part. We
call {\em maximal eigenvector} this eigenvector. Clearly if $H_{y_k}$ is a nonnegative
irreducible matrix, the maximal eigenvector is the Perron vector.
We see in Section~\ref{numer} that this concern is not necessary in practice.

As a starting approximation $y_0$ we may choose the null vector. For close to
critical problems, where $y^\ast$ is close to zero, this choice should
guarantee the convergence, according to the results of Section~\ref{sec:con}.

The resulting iterative process is summarized in Algorithm~1.

\begin{algorithm}[h!t]
\begin{algorithmic}
 \STATE Set $k\leftarrow 0$
 \STATE Set $y_0 \leftarrow 0$
 \STATE Set $w\leftarrow \text{the Perron vector of $b(e,\cdot)+b(\cdot,e)$}$
 \WHILE{$\norm{H_{y_k}y_k-y_k}_1 \geq \varepsilon$}
 \STATE Set $u \leftarrow$ the maximal eigenvector of $H_{y_k}$
 \STATE Compute the normalization factor $\alpha=-\frac{w^T(u-b(u,e)-b(e,u))}{w^Tb(u,u)}$
 \STATE Set $y_{k+1} \leftarrow \alpha u$
  \STATE Set $k\leftarrow k+1$
 \ENDWHILE
 \RETURN $x=e-y_k$
\end{algorithmic}
\caption{The Perron iteration}\label{algpv}
\end{algorithm}

\section{Convergence analysis of the Perron iteration}\label{sec:con}
In this section, we show that the Perron iteration \eqref{eq:PVi} is
locally convergent, and its convergence is linear. Moreover, the convergence
speed gets faster as the problem gets closer to critical.
 
\subsection{Derivatives of eigenvectors}
It is well known \cite{wilkinsonbook} that the eigenvalues and eigenvectors of
a matrix are analytical functions of the matrix entries in a neighborhood
of a simple eigenpair. The following formula for an analytical expression of
their first derivatives is from Meyer and Stewart
\cite[Theorem~1]{meyerstewart}.

\begin{theorem}\label{theoms}
Let $A=A(z)$, $\lambda=\lambda(z)$, $u=u(z)$ be a matrix, eigenvalue and
associated eigenvector depending on a parameter $z \in \mathbb C$. Let us
suppose that $\lambda(z_0)$ is simple and $A'(z_0)$, $\lambda'(z_0)$,
$u'(z_0)$ each exist. Let $w=w(z)$ be another vector such that $w'(z_0)$
exists and let $\sigma(u,w)$ be a function whose value is a real scalar
constant for all $z$. Let $\sigma_1^H$ and $\sigma_2^H$ be the partial
gradients of $\sigma(\cdot,\cdot)$ seen as a function respectively of its
first and second vector argument only.

If $\sigma_1^H u\neq 0$ for $z=z_0$, then the derivative $u'$ of $u$ at $z=z_0$ is given by
\[
 u'=\frac{\sigma_1^H(A-\lambda I)^\ginv A'u-\sigma_2^H w'}{\sigma_1^H u}u-(A-\lambda I)^\ginv A'u.
\]
\end{theorem}
Here $X^\ginv$ denotes the so-called \emph{group inverse} of a singular matrix
$X$, i.e., the inverse of $X$ in the maximal multiplicative subgroup
containing $X$. We refer the reader to the abovementioned paper for more
details on group inverses.

In fact, very little is needed on group inverses, and the formula can be modified slightly in order to replace it with the Moore-Penrose pseudoinverse $X^\dagger$, which is a more canonical tool in matrix computations.
\begin{theorem}\label{theomsmod}
 With the same hypotheses as Theorem~\ref{theoms}, let $v(z)$ be the left
 eigenvector of $A(z)$ corresponding to the eigenvalue $\lambda(z)$. 
If $\sigma_1^H u\neq 0$ for $z=z_0$, then the derivative $u'$ of $u$ at $z=z_0$ is given by
\begin{align}\label{msmod}
  u'&=\frac{\sigma_1^H(A-\lambda I)^\dagger (A'-\lambda' I)u-\sigma_2^H w'}{\sigma_1^H u}u-(A-\lambda I)^\dagger (A'-\lambda' I)u,
\end{align}
with $\lambda' = \frac{v^H A' u}{v^H u}$.
\end{theorem}
\begin{proof}
  The proof is a minor modification of the original proof \cite{meyerstewart}
  of Theorem~\ref{theoms}. By differentiating the identity $Au=\lambda u$ we
  get $A' u+Au'=\lambda' u+ \lambda u'$, i.e.,
\begin{equation} \label{ms0}
 (A-\lambda I)u'=-(A'-\lambda' I)u. 
\end{equation}
By left-multiplying everything by $v^H$, and noting that $v^H A=\lambda v^H$, we get the required expression for the eigenvalue derivative $\lambda'$. Moreover, since $u$ is a simple eigenvector at $z=z_0$, the kernel of $(A-\lambda I)$ is $\operatorname{span}(u)$. Thus from \eqref{ms0} we can determine $u'$ up to a scalar multiple of $u$:
\begin{equation}\label{ms1}
 u'=-(A-\lambda I)^\dagger (A'-\lambda' I)u + \delta u.
\end{equation}
We shall now use the normalization condition $\sigma(u,w)=k$ to determine the value of $\delta$. By differentiating it, we get
\begin{equation}\label{ms2}
 \sigma^H_1(u,w) u' + \sigma^H_2(u,w) w'=0. 
\end{equation}
Plugging \eqref{ms1} into \eqref{ms2} yields a linear equation for $\delta$.
\end{proof}

\subsection{Jacobian of the Perron iteration}
The Perron iteration is a fixed-point iteration for the function
$F(y):=PV(H_y)$, where the function $u=PV(X)$ returns the Perron vector $u$ of
the nonnegative irreducible matrix $X$, normalized such that $w^T(u-H_u u)=0$,
where $w$ is a fixed positive vector. We can use Theorem~\ref{theomsmod} to
compute the Jacobian of this map $F$.

\begin{theorem}
Let $y$ be such that $H_y$ is nonnegative and irreducible.
Let $u=F(y)$, and let
$v$ such that $v^TH_y=\lambda v^T$, where $\lambda=\rho(H_y)$.
Then the Jacobian of the map  $F$ at $y$ is
\begin{equation}\label{jac}
 JF_y=\left(I-\frac{u\sigma_1^T}{\sigma_1^T u}  \right) (H_y-\lambda I)^\dagger  \left(I-\frac{uv^T}{v^Tu}\right)b(\cdot,u)
\end{equation}
where 
\[
\sigma_1^T=w^T(I-b(e-u,\cdot)-b(\cdot,e-u)).
\]
\end{theorem}

\begin{proof}
We shall compute first the directional derivative of $F$ at  $y$ along the
direction $a$. To this purpose, let us set $y(z):=y+az$, for any $z \in \mathbb C$, and
$A(z)=H_y$. We have 
\[
A'(z)=\frac{d}{dz}H_y=-b(a,\cdot).
\]
Moreover, set 
\[
\sigma(u,w)=w^T(u-b(e,u)-b(u,e)+b(u,u)),
\] 
where $w(z)=w$ for each $z$ (so that $w'=0$).
The partial gradient of $\sigma(\cdot,\cdot)$ with respect to the first argument is 
$
\sigma_1^T=w^T(I-b(e-u,\cdot)-b(\cdot,e-u)).
$
Plugging everything into \eqref{msmod}, we get
\[
\begin{aligned}
 u'=&\frac{\sigma_1^T(A-\lambda I)^\dagger (A'-\lambda' I)u}{\sigma_1^T u}u-(A-\lambda I)^\dagger (A'-\lambda' I)u\\
 =& -\left(I-\frac{u\sigma_1^T}{\sigma_1^T u}  \right) (A-\lambda I)^\dagger \left(A'-\frac{v^T A' u}{v^Tu} I\right)u\\
 =& -\left(I-\frac{u\sigma_1^T}{\sigma_1^T u}  \right) (A-\lambda I)^\dagger  \left(I-\frac{uv^T}{v^Tu}\right)A'u\\
=& -\left(I-\frac{u\sigma_1^T}{\sigma_1^T u}  \right) (A-\lambda I)^\dagger  \left(I-\frac{uv^T}{v^Tu}\right)(-b(a,u))\\
=& \left(I-\frac{u\sigma_1^T}{\sigma_1^T u}  \right) (A-\lambda I)^\dagger  \left(I-\frac{uv^T}{v^Tu}\right)b(\cdot,u)a.\\
\end{aligned}
\]
From this expression for the directional derivative, it is immediate to
recognize that the Jacobian is \eqref{jac}.
\end{proof}

\subsection{Local convergence of the iteration}
The fixed-point iteration $y_{k+1}=F(y_k)$ is locally convergent in a
neighborhood of $y^\ast$ if and only if the spectral
radius of $J F_{y^\ast}$ is strictly smaller than 1. First notice that it makes
sense to compute the Jacobian using \eqref{jac} in a neighborhood of the
solution $y^\ast$. In fact,
$\sigma_1^Ty^\ast=w^T(y^\ast-b(e-y^\ast,y^\ast)-b(y^\ast,e-y^\ast))=w^Tb(y^\ast,y^\ast)$
and the latter quantity is positive as $w>0$ and $b(y^\ast,y^\ast)\gneq 0$, as stated in Section~\ref{assume}. Moreover, since $\lambda=1$ is a simple eigenvalue, the left and right
eigenvectors $v=v^\ast$ and $u=y^\ast$ cannot be orthogonal.

By evaluating \eqref{jac} at $y=y^\ast$, we get
\begin{equation}\label{jacsol}
  J F_{y^\ast}=\left(I-\frac{y^\ast\sigma_1^{\ast T}}{w^Tb( y^\ast, y^\ast)}  \right) A^\dagger  \left(I-\frac{y^\ast v^{\ast T}}{v^{\ast T} y^\ast}\right)b(\cdot,y^\ast),
\end{equation}
where we have set $\sigma_1^{\ast T}=w^T(I-b(e-y^\ast,\cdot)-b(\cdot,e-y^\ast))$ and $A=(b(e-y^\ast,\cdot)+b(\cdot,e)-I)$.

Let us try to understand what happens to the spectral radius $\rho(JF_{y^\ast})$ when the problem is close to critical. 

\begin{theorem}
Let $b_t(\cdot,\cdot)$, $t\in [0,1]$ be an analytical one-parameter family of
Markovian binary trees, which is supercritical for $t\in [0,1)$ and critical
  for $t=1$, and let us denote with an additional subscript $t$ the quantities
  defined above for this family of problems. Let us suppose that
  $R_t:=b_t(e,\cdot)+b_t(\cdot,e)$ is irreducible for every $t\in [0,1]$, and
  let $\rho(JF_{y_t^\ast,t})$ be
  the spectral radius  of the Jacobian of the Perron
  iteration as defined in \eqref{jacsol}. Then 
\[
\lim_{t\to1}\rho(JF_{y_t^\ast,t})=\abs{1-\frac{ \widehat v^T_1 b_1(\widehat{y}_1,\widehat{y}_1)  }{w^Tb_1(\widehat{y}_1, \widehat{y}_1)   }\frac{w^T\widehat{y}_1}{\widehat v_1^T \widehat{y}_1}  }
\]
where $\widehat y_1$ and $\widehat v_1^T$ are left and right Perron vectors of $R_1$. As a special case, if the vector $w$ is a scalar multiple of $\widehat v_1$, the limit is 0.
\end{theorem}
\begin{proof}
Let us define $\widehat y_t$
as the Perron vector of $H_{y^\ast_t,t}$ 
normalized so that $\norm{\widehat y_t}_1=1$, and similarly $\widehat v_t^T$
as the left Perron vector of the same matrix, normalized so that
$\norm{\widehat y_t}_1=1$. Since $H_{y^\ast_t,t}$ is irreducible, its left and
right Perron vectors are analytical functions of $t$, and thus $\widehat y_t$
and $\widehat v_t$ converge to $\widehat y_1$ and $\widehat v_1$ respectively.
We have $H_{y^\ast_1,1}=H_{0,1}=R_1$. Notice that $\sigma_{1,t}^{*,T}\to w^T
(I-R_1)$, and that
\[
A_t^\dagger=(b(e-y^\ast,\cdot)+b(\cdot,e)-I)^\dagger \to (R_1-I)^\dagger.
\]
Moreover, since $\widehat y_1$ and $\widehat v_1$ span the right and left kernel of $I-R_1$, we have 
\[
(I-R_1)(I-R_1)^\dagger= I-\frac{\widehat y_1 \widehat v_1^T}{\widehat v_1^T\widehat y_1}.
\]
Additionally, we shall make use of the relation $\rho(AB)=\rho(BA)$, valid for any $A$ and $B$ such that $AB$ and $BA$ are square matrices, in the first and second-to-last step of the following computation.

Putting all together, we get
\begin{align*}
\rho(J F_{y_t^\ast,t})=&\rho\left(\left(I-\frac{y_t^\ast\sigma_{1,t}^{\ast T}}{w^Tb_t( y_t^\ast, y_t^\ast)}  \right) A_t^\dagger  \left(I-\frac{y_t^\ast v_t^{\ast T}}{v_t^{\ast T} y_t^\ast}\right)b_t(\cdot,y_t^\ast)\right)\\
=&\rho\left(\left(b_t(\cdot,y_t^\ast)-\frac{b_t(y_t^\ast,y_t^\ast)\sigma_{1,t}^{\ast T}}{w^Tb_t( y_t^\ast, y_t^\ast)}  \right) A_t^\dagger  \left(I-\frac{y_t^\ast v_t^{\ast T}}{v_t^{\ast T} y_t^\ast}\right)\right)\\
=&\rho\left(\left(b_t(\cdot,y_t^\ast)-\frac{b_t(\widehat
      y_t,\widehat y_t)\sigma_{1,t}^{\ast T}}{w^Tb_t(\widehat  y_t, \widehat
      y_t)}  \right) A_t^\dagger  \left(I-\frac{\widehat y_t \widehat
      v_t^T}{\widehat v_t^T \widehat y_t}\right)\right).
\end{align*}
Therefore, we obtain
\begin{align*}
\lim_{t\to 1}\rho(J F_{y_t^\ast,t})
=&\rho\left(-\frac{b_1(\widehat y_1,\widehat y_1)w^T(I-R_1)}{w^Tb_1(\widehat  y_1, \widehat y_1)} (R_1-I)^\dagger \left(I-\frac{\widehat y_1 \widehat v_1^T}{\widehat v_1^T \widehat y_1}\right)\right)\\
=&\rho\left(\frac{b_1(\widehat y_1,\widehat y_1)w^T}{w^Tb_1(\widehat  y_1, \widehat y_1)} \left(I-\frac{\widehat y_1 \widehat v_1^T}{\widehat v_1^T \widehat y_1}\right)^2
\right)\\
=&\rho\left(\frac{b_1(\widehat y_1,\widehat y_1)w^T}{w^Tb_1(\widehat  y_1, \widehat y_1)} \left(I-\frac{\widehat y_1 \widehat v_1^T}{\widehat v_1^T \widehat y_1}\right) \right)\\
=&\rho\left(\frac{1}{w^Tb_1(\widehat  y_1, \widehat y_1)} w^T \left(I-\frac{\widehat y_1 \widehat v_1^T}{\widehat v_1^T \widehat y_1}\right) b_1(\widehat y_1,\widehat y_1) \right)\\
=&\abs{1-\frac{ \widehat v^T_1 b_1(\widehat{y}_1,\widehat{y}_1)  }{w^Tb_1(\widehat{y}_1, \widehat{y}_1)}\frac{w^T\widehat{y}_1}{\widehat v_1^T \widehat{y}_1}}.\qedhere
\end{align*}
\end{proof}

For the normalization condition, the above result suggests taking $w$ as the
left Perron vector of $b(e,\cdot)+b(\cdot,e)$. Indeed, this choice guarantees
the local convergence of the Perron iteration for close-to-critical problems.
Moreover we point out that, even though the convergence is linear, the speed
of convergence increases as the MBT gets closer to critical; in particular,
the convergence is superlinear in the critical case.

\section{Numerical experiments}\label{numer}
We compared the Perron iteration (PI)  with
the Newton method (NM) \cite{HautphenneLatoucheRemiche} and with the \emph{thicknesses} algorithm (TH) \cite{HLR10}. As stated before, TH and PI are linearly convergent algorithms, while NM is a quadratically convergent one. All the experiments were performed using Matlab 7 (R14) on an Intel Xeon 2.80Ghz bi-processor.

We applied the algorithms to the two test cases reported in
\cite{HautphenneLatoucheRemiche}. The first one (E1) is an MBT of size $n=9$
depending on a parameter $\lambda$, which is critical for $\lambda \approx
0.85$ and supercritical for larger values of $\lambda$. The second one (E2) is
a MBT of size $n=3$ depending on a parameter $\lambda$, which is critical for
$\lambda\approx 0.34$ and $\lambda \approx 0.84$, and supercritical for the
values inbetween.

The only noteworthy issue in the implementation of PI is the method used for
the computation of the maximal eigenvector. The classical methods are usually
optimized for matrices of much larger size; however, here we deal with
matrices of size $n=3$ and $n=9$, for which the complexity constants matter.
We compared several candidates (\texttt{eigs}, \texttt{eig}, the power method,
a power method accelerated by repeated squaring of the matrix), and found that
in our examples the fastest method to find the maximal eigenvector is
computing the full eigenvector basis with \texttt{[V,Lambda]=eig(P)} and then
selecting the maximal eigenvector. The picture should change for
problems of larger size: \texttt{eig} takes $O(n^3)$ operations, while for
instance $\texttt{eigs}$ should take only $O(n^2)$ in typical cases. On the
other hand, we point out that in absence of any structure (such as sparsity) in $b(\cdot,\cdot)$, forming the matrix $b(v,\cdot)$ or $b(\cdot,v)$ for a new vector $v$, an operation which is required at every step in all known iterative algorithms, requires $O(n^3)$ operations. Therefore, the CPU times are somehow indicative of the real complexity of the algorithms, but should be taken with a grain of salt.

The stopping criterion was chosen to be $\norm{x-a+b(x,x)}\leq n \varepsilon$, with $\varepsilon=10^{-13}$, for all algorithms.

The table in Figure~\ref{tabella} shows the results for several choices of $\lambda$. The algorithm TH is clearly the slowest, taking far more CPU time than the two competitors. The different behavior of PI when approaching the critical cases is apparent: while the iterations for TH and NM increase, PI seems to be unaffected by the near-singularity of the problem, and in fact the iteration count decreases slightly.
\begin{figure}
\caption{CPU time in sec. (and number of iterations in brackets) for TH, NM and PI on several choices of $\lambda$ for E1 (top) and E2 (bottom)}\label{tabella}
\begin{tabular}{cccc}
$\lambda$ & TH & NM & PI\\
\hline
0.86 & 2.3935e+00 (11879)& 5.0932e-03 (14)& 4.9267e-03 (7)\\
0.9 & 6.5353e-01 (3005) & 4.2859e-03 (12) & 5.5756e-03 (8)\\
1 & 2.8049e-01 (1149) & 3.9009e-03 (11) & 5.5090e-03 (8)\\
2 & 9.2644e-02 (191) & 2.8453e-03 (8) & 5.5125e-03 (8)\\
\end{tabular}
\begin{tabular}{cccc}
$\lambda$ & TH & NM & PI\\
0.5& 7.7003e-02 (132) & 2.3305e-03 (8) & 5.6983e-03 (11)\\
0.7& 7.6503e-02 (135) & 2.1842e-03 (8) & 5.6081e-03 (11)\\
0.8& 9.3603e-02 (313) & 2.4543e-03 (9) & 4.6166e-03 (9)\\
0.84& 7.3060e-01 (4561) & 4.0001e-03 (13) & 4.1090e-03 (8)\\
\end{tabular}
\end{figure}

To show further results on the comparison between NM and PI, we report a number of graphs comparing the iteration count and CPU times of the two algorithms. The graphs are not cut at the critical values, but they extend
to subcritical cases as well. It is an interesting point to note that when the MBT is subcritical, and thus the minimal solution $x^\ast$ (extinction probability) is $e$, the two algorithms have a different behavior: NM (and TH as well) converges to $e$, while PI skips this solution and converges to a different solution $x>e$. This is because in the derivation of the Perron iteration we chose the solution $\alpha \neq 0$ for the normalization equation, thus explicitly excluding the solution $y=0$
(i.e., $x=e$).

\begin{figure}
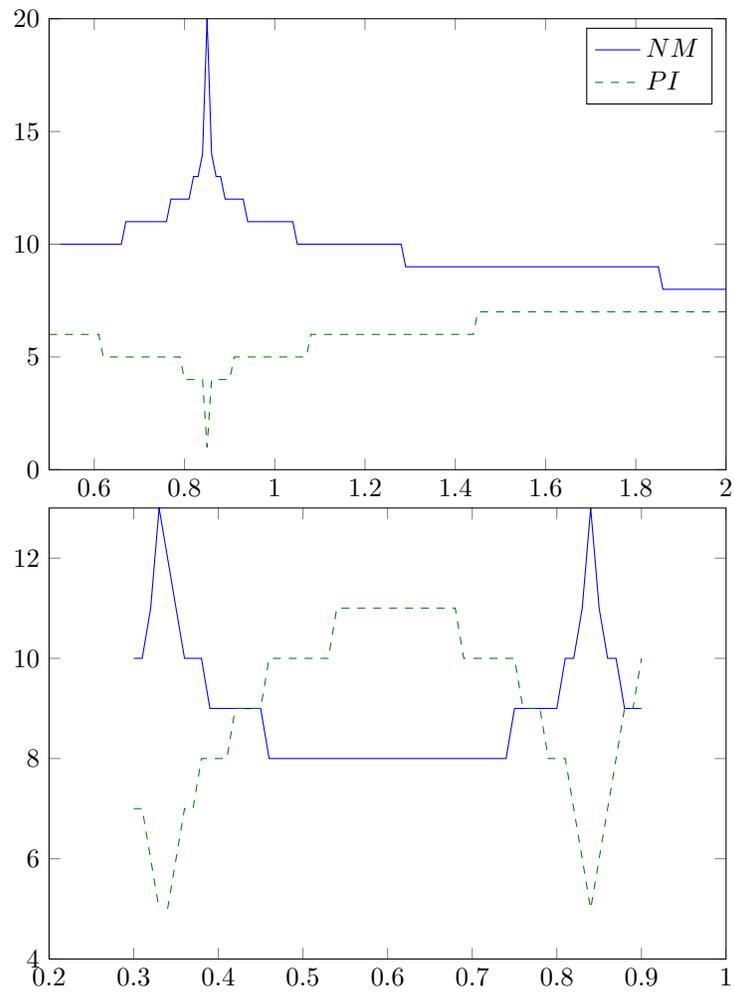

 \begin{center}
   \setlength\plotheight{6cm}
   \setlength\plotwidth{9cm}
  \beginpgfgraphicnamed{E1it}
  \input{experiments/E1it.tex}
  \endpgfgraphicnamed
\\
  \beginpgfgraphicnamed{E2it}
  \input{experiments/E2it.tex}
  \endpgfgraphicnamed
 \end{center}
 \caption{Iteration count vs. parameter $\lambda$ for E1 (top) and E2 (bottom)} \label{it}
\end{figure}

Figure~\ref{it} shows a plot of the iteration count of the two methods vs. different values of the parameter $\lambda$. While in close-to-critical cases the iteration count for NM has a spike,  the one for PI seems to decrease. However, the iteration count comparison is not fair since the steps of the two iterations require a different machine time.
\begin{figure}
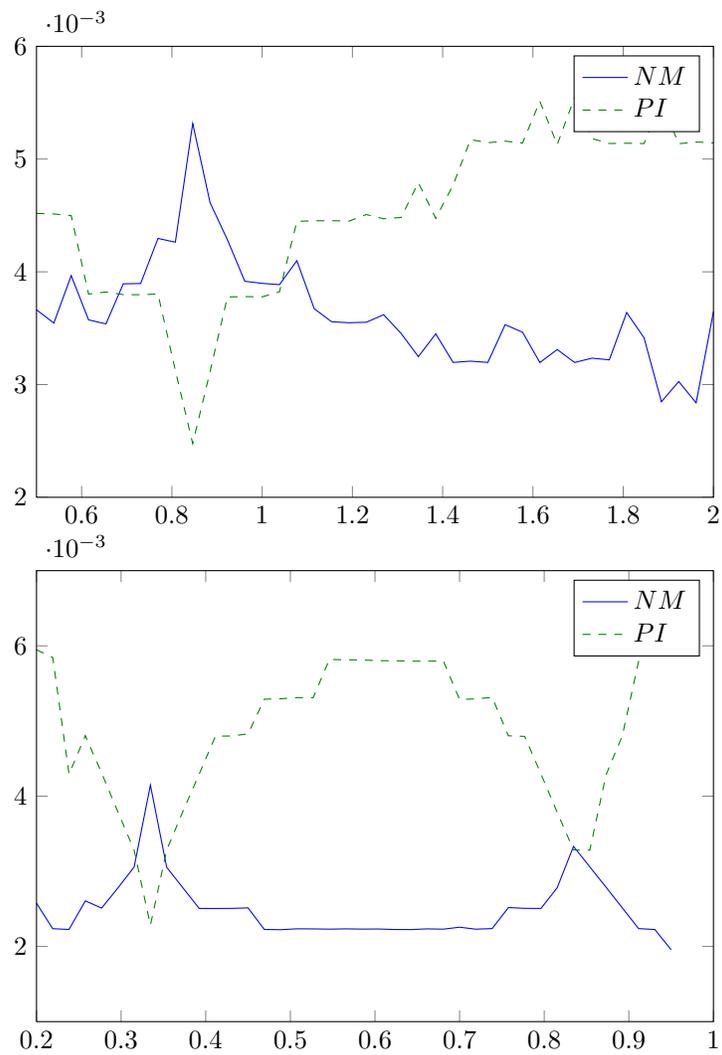

 \begin{center}
   \setlength\plotheight{6cm}
   \setlength\plotwidth{9cm}
  \beginpgfgraphicnamed{E1tm}
  \input{experiments/E1tm.tex}
  \endpgfgraphicnamed
\\
  \beginpgfgraphicnamed{E2tm}
  \input{experiments/E2tm.tex}
  \endpgfgraphicnamed
 \end{center}
 \caption{CPU time (in sec.) vs. parameter $\lambda$ for E1 (top) and E2 (bottom)} \label{tm}
\end{figure}
Figure~\ref{tm} shows a similar plot, considering the CPU time instead of the iteration count. In order to achieve better accuracy, the plotted times are averages over 100 consecutive runs.

The results now favor the Newton method in most experiments, but in close-to-critical cases the new method achieves better performance. The results are very close to each other, though, so it is to be expected that for larger input sizes or different implementations the differences in the performance of the eigensolver could lead to significant changes in the results.

\begin{figure}
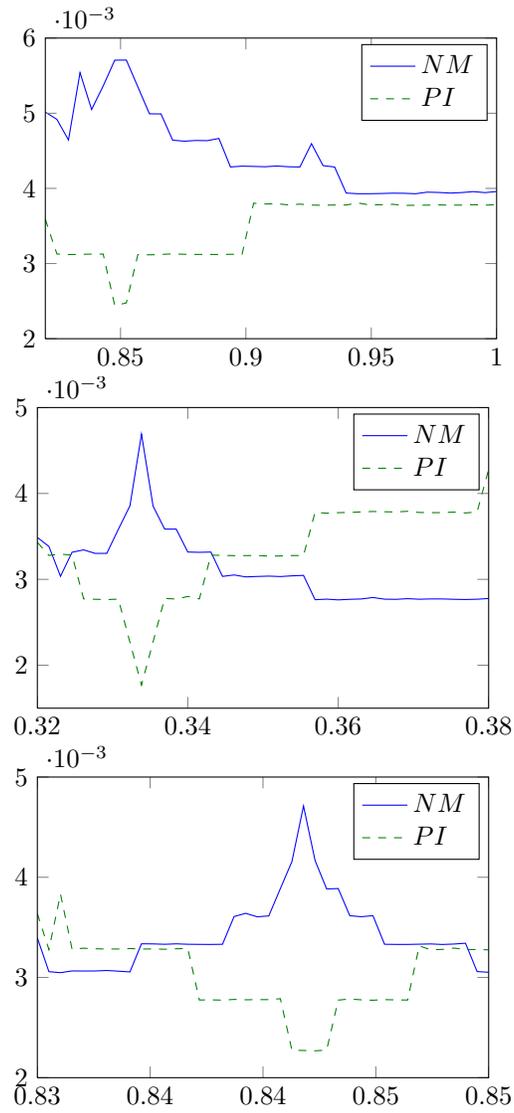

 \begin{center}
   \setlength\plotheight{4cm}
   \setlength\plotwidth{6cm}
  \beginpgfgraphicnamed{E1zm}
  \input{experiments/E1zm.tex}
  \endpgfgraphicnamed
\\
  \beginpgfgraphicnamed{E2zm1}
  \input{experiments/E2zm1.tex}
  \endpgfgraphicnamed
\\
  \beginpgfgraphicnamed{E2zm2}
  \input{experiments/E2zm2.tex}
  \endpgfgraphicnamed
 \end{center}
 \caption{Detailed views from Figure~\ref{tm}: CPU time (in sec.) vs. parameter $\lambda$ for E1 (top) and E2 (middle and bottom)} \label{big}
\end{figure}

In order to highlight the performance difference in close-to-critical cases, we report in Figure~\ref{big} a plot with the CPU times sampled at a larger number of points around the most ``interesting'' regions of the previous graphs.

The Jacobian \eqref{jacsol} had spectral radius less than 1 in all the above experiments, a condition which is needed to ensure the convergence of PI. However, this is not true for all possible MBTs. In fact, by setting the parameter $\lambda$ for E1 to much larger values, we encountered problematic cases in which PI did not converge. Specifically, starting from $\lambda \approx 78$ the Jacobian \eqref{jacsol} is larger than 1 and PI does not converge. However, such cases are of little practical interest since they are highly supercritical MBTs, distant from the critical case, and thus they are easily solved with the traditional methods (NM or the customary functional iterations \cite{BeanKontoleonTaylor}) with a small number of iterations. 

The problem E2 is well-posed only for $0 \leq \lambda \leq 1$, otherwise negative entries appear in $b$, thus the above discussion does not apply.

Along all the experiments reported above, all the matrices $H_y$ appearing in the PI steps always turned out to have positive entries, even in the subcritical problems; thus their Perron vector and values were always well-defined and real.

\section{Conclusions and open issues}\label{conclude}
We have proposed a new algorithm for solving the quadratic vector
equation \eqref{eq:qve}, based on a Perron iteration. The algorithm performs
well, both in terms of speed of convergence and accuracy, for
close-to-critical problems where the classical methods are slower.

Along the framework that we have exposed, several different choices are possible in the
practical implementation of the new algorithm. 

One of them is the choice of the bilinear form $b(\cdot,\cdot)$. Equation \eqref{eqy}
and its solution depend only on the quadratic form $b(t,t)$; however, there
are different ways to extend it to a bilinear form $b(s,t)$. This choice
ultimately reflects a modeling aspect of the problem: when an individual
spawns, it is transformed into two individuals in different states, and we may
choose arbitrarily which of them is called the \emph{mother} and which
 the \emph{child}.

As an example of how this choice affects the solution algorithms, 
changing the bilinear form may transform the \emph{depth}
algorithm into the \emph{order} one and vice versa.  The algorithms we
proposed depend on the actual choice of the bilinear extension of the
quadratic form $b(t,t)$, and the convergence speed is affected by this
decision.

A second choice is the normalization of the computed Perron vector: different
approaches may be attempted --- for instance, minimization of the $1$-norm, of the $2$-norm,
or orthogonality of the residual of
\eqref{eq:pv} with respect to a suitably chosen vector --- although it is not clear whether we can improve the results of the normalization presented here.

A third choice, crucial in the computational experiments, is the method used
to compute the Perron vector. For moderate sizes of the problem, it is cheaper
to do a full eigendecomposition of the matrix and extract the eigenvalue with
maximum modulus, but for larger problems it pays off to use different specific
methods for its computation. 

All these variants deserve to be better understood, and are now under our
investigation.

\bibliographystyle{abbrv}
\bibliography{qve}
\end{document}